\documentclass[12pt]{article}

\usepackage{amsmath}
\usepackage{amsthm}
\usepackage{textcomp}

\usepackage{geometry}                
\usepackage[parfill]{parskip}    
\usepackage{graphicx}
\usepackage{amssymb,latexsym} 
\usepackage{epstopdf}

\newtheorem{theorem}{Theorem} 
\newtheorem{lemma}[theorem]{Lemma}   

\newtheorem{definition}[theorem]{Definition}

\newtheorem{corollary}[theorem]{Corollary}

\newcommand{\N}{\mathbb{N}}

\newcommand{\eps}{\epsilon}
\newcommand{\goesto}[1]{\xrightarrow{#1}}

\newcommand{\cross}{\times}

\newcommand{\Lam}{\mathcal{L}}

\newcommand{\ds}{\displaystyle}

\title{An Alternative Approach to Extending pseudo-Anosovs Over Compression Bodies}
\author{R. Ackermann}

\begin{document}

\maketitle

\begin{abstract}
A recent paper (\cite{BJM}) by Biringer, Johnson, and Minsky prove that any pseudo-Anosov whose stable lamination is the limit of disks in a compression body has a power which extends over some non-trivial minimal compression body.  This paper presents an alternative proof of their theorem, using techniques of Long and Casson which first appeared in~\cite{CL} and ~\cite{L}.  The key ingredient is the existence of a certain collection of disks whose boundaries are formed from an arc of the stable lamination and an arc of the unstable lamination.  Furthermore, the proof here shows that there are only finitely many minimal compression bodies over which a power of a pseudo-Anosov can extend.
\end{abstract}

\section{Introduction}

In~\cite{BJM}, Biringer, Johnson, and Minsky prove the following theorem:
\\
\begin{theorem}
Let $\varphi : F \rightarrow F$ be a pseudo-Anosov with stable lamination $\Lam^+$ and unstable lamination $\Lam^-$.  Say also that a lamination $K^+ \supseteq \Lam^+$ bounds in a compression body $M$ and $M$ is minimal with respect to this condition.  

Then there exists $k$ such that $\varphi^k$ extends over $M$.

\end{theorem}

Here we say that a lamination bounds if it is the Hausdorff limit of curves bounding disks in the compression body and that a compression body $M$ is minimal with respect to $K^+$ bounding if there is no inequivalent $N \subset M$ in which $K^+$ bounds.  Their proof makes use of relatively recent ideas including $\delta$-hyperbolic geometry, the curve complex, and Ahlfors-Bers theory.  They also give examples which show that their theorem if false if $\varphi^k$ is replaced with $\varphi$ in the conclusion.

The purpose of this paper is to offer an alternative proof to this theorem using older ideas first introduced by Casson and Long.  More specifically, in~\cite{CL} Casson and Long provide an algorithm for determining whether a particular pseudo-Anosov extends over some compression body and in~\cite{L} Long goes on to show that a pair of minimal, transverse laminations can bound in only finitely many compression bodies.

The proof given here is achieved by generalizing lemmas of Casson and Long.  The basic idea is to show that disks of a particular type must exist in any compression body in which the stable lamination of $\varphi$ bounds and some curve approximating the unstable lamination bounds as well.  Using these disks, we build a non-empty but finite collection of compression bodies over which $\varphi$ could potentially extend.  Within this collection there is a (possibly smaller) collection which is invariant under the action of $\varphi$, implying that a power of $\varphi$ extends.

\section{Definitions and basic facts}

Let $F$ be a closed, orientable surface of genus at least two.  A \emph{compression body} is any three manifold formed by taking $F \cross I$, attaching disjoint 2-handles to the boundary surface $F \cross \{1\}$, and filling in any resulting 2-spheres with 3-handles.  The boundary surface $F \cross \{0\}$ is called the \emph{exterior surface} of $M$.  Call $F \cross I$ the \emph{trivial compression body}.

A compression body $M$ with exterior surface $F$ has associated to it a normal subgroup  $\ds N = \mathrm{ker} \left( i_* : \pi_1\left(F\right) \rightarrow \pi_1\left(M\right)\right)$ where $i : F \rightarrow M$ is inclusion.  Call $N$ the \emph{planar kernel} of $M$, and say two compression bodies are \emph{equivalent} if they have isomorphic planar kernels.  If $M_1$ and $M_2$ are two equivalent compression bodies with exterior surface $F$, then anytime a curve in $F$ bounds a disk in $M_1$ there is an isotopic curve in $F$ which bounds a disk in $M_2$.

A \emph{geodesic lamination} $\Lam$ on $F$ is a closed subset which can be written as the disjoint of union of geodesic leaves.  A geodesic lamination is \emph{minimal} if the closure of any leaf is the whole lamination, and a geodesic lamination \emph{fills} if each component of $F \setminus \Lam$ is simply connected.  Call the closure of these components the \emph{complementary regions} of $\Lam$,  and say $\Lam$ is \emph{maximal} if every complementary region is an ideal trigon.  To save words, a lamination will always be assumed to be a geodesic lamination.

Given a closed surface $F$, the \emph{Hausdorff metric} is a metric on all closed subsets of $F$.  For two closed subsets $A$ and $B$, distance is defined by $d_H\left(A, B\right) \leq \eps$ if there are regular neighborhoods $N_\eps\left(A\right) \subseteq B$ and $N_\eps\left(B\right) \subseteq A$.  The Hausdorff metric is particularly useful for measuring how close a simple closed curve is to a minimal lamination.

A surface automorphism $\varphi : F \rightarrow F$ is called \emph{pseudo-Anosov} if it preserves a pair of transverse measured laminations $\left(\Lam^+, \mu_+\right)$ and $\left(\Lam^-, \mu_-\right)$, called the stable and unstable lamination respectively.  In this case, $\Lam^+$ and $\Lam^-$ both are minimal and filling.  We say that $\varphi$ \emph{extends over a compression body} $M$ if there is an automorphism $\psi : M \rightarrow M$ such that $\psi |_F = \varphi$ where $F$ is the exterior surface of $M$.  A necessary and sufficient condition for an automorphism to extend is that the induced isomorphism $\varphi_* : \pi_1\left(F\right) \rightarrow \pi_1\left(F\right)$ leaves invariant the planar kernel of $M$ (see~\cite{CG}, lemma 5.2).

An important fact about the behavior of pseudo-Anosovs is that they exhibit source-sink dynamics on the space of all measured laminations.  In particular, say $\varphi$ is a pseudo-Anosov with invariant laminations $\Lam^+$ and $\Lam^-$.  If $\Lam^+$ and $\Lam^-$ are maximal and $\Lam$ is a third lamination, then the sequence $\{ \phi^k\left(\Lam\right) \}$ converges to $\Lam^+$ and $\{ \phi^{-k}\left(\Lam\right) \}$ converges to $\Lam^-$ in the Hausdorff topology (see~\cite{CB}).

Crucial to our discussion is a definition from~\cite{CL}:
\\
\begin{definition}\label{D:Bounds}
Let $\Lam$ be a geodesic lamination in the exterior surface of a compression body $M$.  Then $\Lam$ \emph{bounds} in $M$ if there is a sequence of simple closed curves $\{ C_i \}$ all of which bound disks in $M$ such that $C_i \rightarrow \Lam$ as $i \rightarrow \infty$
\end{definition} 

Here convergence is meant to be in the Hausdorff metric.  However, if $\Lam^+$ and $\Lam^-$ are transverse, minimal, and maximal measured laminations with full support then $\Lam^+$ bounds after isotopy if and only if there is a sequence of essential simple closed curves $\{C_i\}$ all of which bound disks such that $\mu_- \left(C_i \right) \rightarrow 0$ as $i \rightarrow \infty$.  In~\cite{BJM}, a similar notion of bounding is expressed in terms of limit sets.  One motivation for this definition is that if a pseudo-Anosov extends then both its stable and unstable lamination must bound by source-sink dynamics.

If $\mathcal{L}$ is a lamination, then we say that $M$ \emph{is minimal with respect to } $\mathcal{L}$ \emph{bounding} if $\mathcal{L}$ bounds in $M$ and if $N \subset M$ is a compression body inequivalent to $M$ then the $\mathcal{L}$ does not bound in $N$.

\section{Disks in compression bodies}

Throughout take $\left(\Lam^+, \mu_+\right)$ and $\left(\Lam^-, \mu_-\right)$ to be transverse, minimal, and maximal measured laminations.  The purpose of this section is to show the existence of certain disks in any compression body for which $\Lam^+$ bounds and a curve $C$ close to $\Lam^-$ bounds as well.  We begin by stating a lemma first proved in~\cite{L}.
\\
\begin{lemma}\label{L:LamIntersect}
Let $\eps > 0$ be given.  Then there are numbers $M\left(\eps\right)$, $m\left(\eps\right)$ such that:
\begin{enumerate}

\item If $\alpha^+ \subseteq \Lam^+$, $\alpha^- \subseteq \Lam^-$ are arcs with $\mu_-\left(\alpha^+\right) > M$ and $\mu_+\left(\alpha^-\right) > \eps$, then \\ $\mathrm{int}~ \alpha^+ \cap \mathrm{int}~ \alpha^- \neq \emptyset$.

\item If $\alpha^+ \subseteq \Lam^+$, $\alpha^- \subseteq \Lam^-$ are arcs with $\mu_+\left(\alpha^-\right) < \eps$ and $| \mathrm{int}~ \alpha^+ \cap \mathrm{int}~ \alpha^- | \geq 2$, then $\mu_-\left(\alpha^+\right) > m$

\end{enumerate}

Furthermore, $M\left(\eps\right), m\left(\eps\right) \goesto~ \infty$ as $\eps \goesto~ 0$.
\end{lemma}

Let $N_\delta\left(\Lam^-\right)$ be a closed regular $\delta$-neighborhood of $\Lam^-$.  Such a neighborhood can be foliated by intervals so that it has the structure of a product, and each tie $t$ can be thought of as an arc in $F$ transverse to $\Lam^-$.  Let $\ds r = r\left(\delta\right) = \max \{ \mu_-\left(t\right) ~|~ t \mbox{ is a tie of } N_\delta\left(\Lam^-\right) \}$.
\\
\begin{lemma}\label{L:CurveIntersect}
Let $C$ be a geodesic simple closed curve such that $d_H\left(C, \Lam^-\right) < \delta$ for some small $\delta > 0$ and say $\{ A_n \}$ is a sequence of simple closed geodesics converging to $\Lam^+$.  Then for any $\eps > 0$, there is $N$ such that for all $n \geq N$ we have:
\begin{enumerate}

\item If $\alpha^+ \subseteq A_n$, $\alpha^- \subseteq C$ are arcs with $\mu_-\left(\alpha^+\right) > 2M$ and $\mu_+\left(\alpha^-\right) > \eps$, then \\ $\mathrm{int}~ \alpha^+ \cap \mathrm{int}~ \alpha^- \neq \emptyset$

\item If $\alpha^+ \subseteq A_n$, $\alpha^- \subseteq C$ are arcs with $\mu_+\left(\alpha^-\right) < \eps$ and $| \mathrm{int}~ \alpha^+ \cap \mathrm{int}~ \alpha^- | \geq 2$ then $\mu_-\left(\alpha^+\right) > m - 2r$, where $r = r\left(\delta\right)$ as above.

\end{enumerate}
\end{lemma}

\begin{proof}

Given an arc $\alpha^- \subseteq C$, shrink it slightly and assume its endpoints are on leaves of $\Lam^+$ (without changing $\mu_+\left(\alpha^-\right)$).  Then since $\delta > 0$ is small, we can slide its endpoints along leaves of $\Lam^+$ to obtain a nearby arc $\beta' \subseteq \Lam^-$ with $\mu_+\left(\beta'\right) = \mu_+\left(\alpha^-\right)$.  Shrink $\beta'$ slightly to obtain an arc $\beta$.

Choose $N$ so that for all $n \geq N$ the curve $A_n$ satisfies:

\begin{enumerate}

\item Lemma~\ref{L:LamIntersect} holds with arcs of $A_n$ in place of arcs in $\Lam^+$.


\item If $\alpha^- \subseteq C$ is an arc and $\beta \subseteq \Lam^-$ is chosen as above, then for any $p \in \mathrm{int}~\beta~\cap~A_n$ there is an arc $\phi \subseteq A_n$ with endpoints on $\beta$ and $\alpha^-$ (one of these is $p$) with $\mu_-\left(\phi\right) < r = r\left(\delta\right)$.

\end{enumerate}

These conditions can always be satisfied because the angles between nearby geodesics are close in a lamination (see~\cite{CB}), and because the measure of an arc is preserved under homotopy respecting the leaves of $\Lam^-$.

To prove the first conclusion, let $n \geq N$ and take any arc $\alpha^+ \subseteq A_n$ with $\mu_-\left(\alpha^+\right) > 2M$.  Say for contradiction that there is an arc $\alpha^- \subseteq C$ with $\mu_+\left(\alpha^-\right) > \eps$ such that $\ds \mathrm{int}~\alpha^+~\cap~\mathrm{int}~\alpha^- = \emptyset$.  Take $\beta \subseteq \Lam^-$ as above and note that by condition (1) ~$\alpha^+$ must intersect $\beta$ at least twice.  Thus by condition (2) both endpoints of $\alpha^+$ must lie on short arcs of $A_n$ with endpoints on $\mathrm{int}~\beta$ and $\mathrm{int}~\alpha^-$.  But then by shrinking $\alpha^+$ and allowing $\mu_-\left(\alpha^+\right)$ to change by at most $2r < M$ we obtain an arc which does not intersect $\beta$, a contradiction.  Thus $\mathrm{int}~\alpha^+ \cap \mathrm{int}~\alpha^- \neq \emptyset$.

For the second conclusion, again fix $n \geq N$ and say $\alpha^- \subseteq C$, $\alpha^+ \subseteq A_n$ are arcs with $\mu_+\left(\alpha^-\right) < \eps$ and $| \mathrm{int}~ \alpha^+ \cap \mathrm{int}~ \alpha^- | \geq 2$.  Again take $\beta \subseteq \Lam^-$ as above.  Then by condition (2) $\alpha^+$ can be extended to an arc that intersects $\mathrm{int}~\beta$ at least twice with $\mu_-$-measure at most $2r$ more than $\mu_-\left(\alpha^+\right)$.  Thus by lemma~\ref{L:LamIntersect}, $\mu_-\left(\alpha^+\right) > m - 2r$.

\end{proof}

Using lemma~\ref{L:CurveIntersect}, we gain control over what disks arise in certain types of compression bodies.  The next lemma mirrors the proof of lemma 2.4 in~\cite{L}.
\\
\begin{lemma}\label{L:BoundsDisk}
Let $\delta > 0$ be small and say $C$ be a geodesic simple closed curve with $d_H\left(C, \Lam^-\right) < \delta$.  Say in addition that $C$ and $\Lam^+$ both bound in a compression body $M$.  

Then for any small $\eps > 0$ there are arcs $\alpha^+ \subseteq \Lam^+$, $\alpha^- \subseteq \Lam^-$ such that $\alpha^+ \cup \alpha^-$ is isotopic to the boundary of a disk and:

\begin{enumerate}

\item $\mu_+\left(\alpha^-\right) \leq \eps$

\item $m\left(2\eps\right) - 2r \leq \mu_-\left(\alpha^+\right) \leq 2M\left(\eps\right) + 2r$, where $r = r\left(\delta\right)$

\end{enumerate}
\end{lemma}

\begin{proof}
Since $\Lam^+$ bounds, there is a sequence of curves $\{A_n\}$ all bounding disks in $M$ and converging to $\Lam^+$.  Choose $N$ as in lemma~\ref{L:CurveIntersect} and let $A = A_n$ for some $n \geq N$.  Let $D^-$ be the disk with boundary $C$, and say $D^+$ is the disk with boundary $A$.  We assume $\eps$ is very small compared to the $\mu_+$-measure of $C$.

After isotopy, $D^+ \cap D^-$ is a collection of arcs with endpoints on $A \cap C$.  Say an arc $\gamma$ on $C$ or $A$ contains a \emph{complete set} if whenever one endpoint of an arc in $D^+ \cap D^-$ is on $\gamma$ the other endpoint is on $\gamma$ as well.  Choose an arc $\gamma^+ \subseteq A$ such that $\gamma^+$ is complete, $\mu_-\left(\gamma^+\right) \geq 2M$, and $\gamma^+$ is minimal with respect to these conditions.  By lemma~\ref{L:CurveIntersect}, $\gamma^+$ exists and intersects $C$ many times.

Choose an arc $\phi \subseteq D^+ \cap D^-$ such that $\phi$ has endpoints on $\mathrm{int}~\gamma^+$ and $\phi$ is an outermost such arc in $D^-$.  Then choose $\gamma^- \subseteq C = \partial D^-$ to be the arc with endpoints equal to the endpoints of $\phi$ and with the property that $\mathrm{int}~\gamma^- \cap \mathrm{int}~\gamma^+ = \emptyset$.  By lemma~\ref{L:CurveIntersect}, $\mu_+\left(\gamma^-\right) < \eps$.

Now set $\beta^+ \subseteq A$ to be the sub-arc of $\gamma^+$ having endpoints $\partial \gamma^-$ (which equals $\partial \phi$), and note $\mathrm{int}~\beta^+~\cap~\mathrm{int}~\gamma^- = \emptyset$.  We can stretch $\beta^+$ and $\gamma^-$ a small amount so that $| \mathrm{int}~\beta^+~\cap~\mathrm{int}~\gamma^- | \geq 2$ and hence $\mu_-\left(\beta^+\right) \geq m\left(2 \eps\right) - 2r$.

Since $\partial \beta^+ = \partial \phi \subseteq \mathrm{int}~\gamma^+$, the arc $\beta^+$ is a complete, proper sub-arc of $\gamma^+$.  Thus $\mu_-\left(\beta^+\right) \leq 2M$ by minimality of $\gamma^+$.



Now slide $\gamma^-$ along the leaves of $\Lam^+$ to an arc $\alpha^-$ with $\mu_+\left(\alpha^-\right) = \mu_+\left(\gamma^-\right) < \eps$ and endpoints on leaves of $\Lam^+$.  Isotopic to $\gamma^+$ is an arc $\alpha^+ \subseteq \Lam^+$ with $\partial \alpha^+ = \partial \alpha^-$ and $| \mu_-\left(\gamma^+\right) - \mu_-\left(\alpha^+\right) | \leq r$.  The curve $\alpha^+ \cup \alpha^-$ is essential because it is the union of geodesic arcs, and is isotopic to the boundary of the disk formed by gluing pieces of $D^+$ and $D^-$ cut out by $\phi$.

\end{proof}


\section{Finitely many minimal compression bodies}

Recall that a compression body $M$ is \emph{minimal} with respect to a collection of laminations $\mathcal{A}$ if every element of $\mathcal{A}$ bounds in $M$ and whenever $N \subseteq M$ is a compression body for which every element of $\mathcal{A}$ bounds the planar kernels of $M$ and $N$ are isomorphic.  Also note that a single geodesic simple closed curve is a lamination.

The following lemma was first proved in~\cite{L2}:
\\
\begin{lemma}\label{L:FinForCollection}
Let $\mathcal{C}$ be any finite collection of simple closed curves.  Then there are at most finitely many compression bodies which are minimal with respect to $\mathcal{C}$.
\end{lemma}

The next lemma gives even more control over what compression bodies can contain a specified collection of disks (this is referred to as a ``folklore lemma'' in~\cite{L}):
\\
\begin{lemma}\label{L:FinForContainment}
Let $\{M_1, ..., M_k\}$ be a collection of pairwise inequivalent compression bodies with $M_1 \subset M_2 \subset ... \subset M_k$ (inclusions are strict).  

Then there is an integer $P$, depending only on the genus of $F$, so that $k \leq P$.
\end{lemma}


The following is a generalization of lemmas first proved in~\cite{CB} and~\cite{L} and is one of the key ingredients in proving our main result.
\\
\begin{lemma}\label{L:CollectionsToFinite}
Let $\{\mathcal{C}_i\}$ be a sequence of finite collections of essential simple closed curves such that any sequence $\{C_i~|~C_i \in \mathcal{C}_i\}$ converges to a lamination $\Lam$.  Let $\mathcal{M}$ be the collection of all pairwise inequivalent compression bodies minimal with respect to $\Lam$ and in which a sequence $\{C_i~|~C_i \in \mathcal{C}_i\}$ bounds.



Then $\mathcal{M}$ is finite.  
\end{lemma}


\begin{proof}
Let $\mathcal{P}$ be the collection of all pairwise inequivalent compression bodies $M$ which are minimal with respect to some finite (or empty) collection $\{C_1, ..., C_n~|~C_i \in \mathcal{C}_i\}$ and either $\Lam$ does not bound in $M$ or if it does then $M$ is minimal.  We consider $\mathcal{P}$ as a partially ordered set with $N \leq M$ if $N \subseteq M$.  Note that the trivial compression body is the unique least element of $P$.  To save on notation, let $\Delta_n$ formally denote a collection $\{C_1, ..., C_n~|~C_i \in \mathcal{C}_i\}$.

Any $M \in \mathcal{M}$ must be minimal with respect to some $\{C_i~|~C_i \in \mathcal{C}_i\}$, for any compression body in which such a sequence bounds has $\Lam$ bounding as well.  Thus, by lemma~\ref{L:FinForContainment}, any $M \in \mathcal{M}$ must be minimal with respect to some finite collection $\Delta_n$ and hence $\mathcal{M} \subseteq \mathcal{P}$. 


We show that $\mathcal{P}$ is finite.  By lemma~\ref{L:FinForContainment}, any chain $M_1 \subset M_2 \subset ...$ in $\mathcal{P}$ is finite, so it only remains to show that for every $M$ there are finitely many $N$ such that whenever $M \subseteq X \subseteq N$ we have $X = M$ or $X = N$.  Call such an $N$ a \emph{direct descendant} of $M$.

Say $M \in \mathcal{P}$ is minimal for some $\Delta_n$ and that $\Lam$ does not bound in $M$.  Then there is a minimal $R \in \N$ such that no collection $\Delta_R \supset \Delta_n$ bounds in $M$.  By lemma~\ref{L:FinForCollection}, there are only finitely many compression bodies minimal with respect to a collection $\Delta_l$ with $l \leq R$.  Any direct descendant of $M$ must be minimal with respect to one of these collections, and thus $M$ has only finitely many direct descendants.


Now say $M \in \mathcal{P}$ such that an infinite sequence $\{C_i~|~C_i \in \mathcal{C}_i\}$ bounds.  Then $\Lam$ bounds in $M$, and by minimality $M$ has no direct descendants in this case.  Thus $\mathcal{P}$ is finite, and $\mathcal{M}$ is finite as well.

\end{proof}




Given a compression body $M$, let $C_1, ..., C_n$ be disjoint curves to which 2-handles are attached to form $M$ from $F \cross I$.  For any automorphism $\varphi$, define $\varphi M$ to be the compression body formed by attaching 2-handles along $\varphi C_1, ..., \varphi C_n$.  Note that $\varphi$ extends over $M$ if and only if $\varphi M$ is equivalent to $M$.
\\
\begin{lemma}\label{L:Nonempty}
Let $\varphi$ be a pseudo-Anosov with maximal stable, unstable laminations $\Lam^+$, $\Lam^-$ and say that $\Lam^+$ bounds in some compression body.  Let $\delta > 0$ and let $\mathcal{N}$ be the collection of all compression bodies minimal with respect to $\Lam^+$ and which have a disk $D$ with $d_H\left(\partial D, \Lam^-\right) < \delta$.

Then $\mathcal{N}$ is non-empty.

\end{lemma}

\begin{proof}
Let $M$ be any compression body minimal for $\Lam^+$ and say $D$ is a disk in $M$.  By the source-sink dynamics of pseudo-Anosovs, for some $k$ the curve $\varphi^{-k}\left(\partial D\right)$ is, after isotopy, a geodesic simple closed curve with $d_H\left(\varphi^{-k}\left(\partial D\right), \Lam^-\right) < \delta$.

Now, let $\{ C_i \}$ be a sequence of curves bounding disks in $M$, such that $\{ C_i \}$ approaches $\Lam^+$ in the Hausdorff topology.  Then $\{ \varphi^{-k}\left(C_i\right) \}$ has the same properties in the compression body $\varphi^{-k} M$.  Furthermore, $\varphi^{-k} M$ is still minimal for if $N \subseteq \varphi^{-k} M$ such that $\Lam^+$ bounds, then $\varphi^{k}N \subseteq M$ and in fact $\varphi^{k} N$ is equivalent to $M$ by minimality.

\end{proof}

Finally, we prove the main theorem.
\\
\begin{theorem}
Let $\varphi : F \rightarrow F$ be a pseudo-Anosov with stable lamination $\Lam^+$.  Assume that $\Lam^+$ is maximal, that $\Lam^+$ bounds in a compression body $N$, and that $N$ is minimal with respect to this condition.

Then there exists $k$ such that $\varphi^k$ extends over $N$.
\end{theorem}

\begin{proof}

Let $\delta > 0$ be small and choose a monotonically decreasing sequence $\{\eps_i\}$ with $\eps_1 > 0$ and $\eps_i \rightarrow 0$ as $i \rightarrow \infty$.  Define $\mathcal{C}_i$ to be the collection of all simple closed curves $\alpha^+ \cup \alpha^-$ formed from arcs $\alpha^+ \subseteq \Lam^+$ and $\alpha^- \subseteq \Lam^-$ where $m\left(2\eps_i\right) - 2r \leq \mu_-\left(\alpha^+\right) \leq 2M\left(\eps_i\right) + 2r$ and $\mu_+\left(\alpha^-\right) \leq \eps_i$ (here $\Lam^-$ is the unstable lamination of $\varphi$ and $M$, $m$, and $r = r\left(\delta\right)$ are as in lemma~\ref{L:CurveIntersect}).  After identifying isotopic curves, each $\mathcal{C}_i$ is finite and any sequence $\{ C_i~|~C_i \in \mathcal{C}_i \}$ converges to $\Lam^+$.

Now let $\mathcal{N}$ be the collection of all compression bodies which are minimal for $\Lam^+$ and also have a sequence of curves $\{C_i~|~C_i \in \mathcal{C}_i\}$ all of which bound disks.  By lemma~\ref{L:CollectionsToFinite}, $\mathcal{N}$ is finite.

Let $\mathcal{N}'$ be the collection of compression bodies which are minimal for $\Lam^+$ and also have a disk $D$ with $d_H\left(\partial D, \Lam^-\right) < \delta$.    By lemma~\ref{L:Nonempty}, the set $\mathcal{N}'$ is nonempty and by lemma~\ref{L:BoundsDisk} it is contained in $\mathcal{N}$.  Applying the techniques of lemma~\ref{L:Nonempty} once again shows that $\varphi^{-t} \mathcal{N}' \subseteq \mathcal{N}'$ for some $t$, and thus there is a subset of $\mathcal{N}$ invariant under the action of $\varphi^{-1}$.  Call this collection $\mathcal{N}^*$.

Now, $\varphi^{-s} N$ lies in $\mathcal{N}^*$ and thus for some $k$ we have $\varphi^{-s - k} N = \varphi^{-s} N$.  Composing with $\varphi^{s + k}$ we have $N = \varphi^k N$ and so $\varphi$ extends over $N$.


\end{proof}

The above proof implies the following corollary, though it also follows from results of~\cite{CL}.
\\
\begin{corollary}
Let $\varphi$ be a pseudo-Anosov with maximal invariant laminations.  

Then $\varphi$ extends over at most finitely many compression bodies minimal with respect to the condition that the stable lamination of $\varphi$ bounds.
\end{corollary}

{\bf Remark:}  The observant reader will note that the theorem above is not quite the same as the theorem of Biringer, Johnson, and Minsky as we have added the hypothesis that the invariant laminations of $\varphi$ are maximal.  If they are not, it is necessary to consider a finite collection of laminations which are formed from $\Lam^+$ and $\Lam^-$ by adding isolated leaves which ``cut across'' the diagonals of principal regions (see~\cite{CB}).  Lemmas~\ref{L:CurveIntersect}, ~\ref{L:BoundsDisk}, and~\ref{L:Nonempty} can be modified to take into account this situation, however it makes their statements and proofs far more clumsy so we do not do so here.






\begin{thebibliography}{CFP2}

\bibitem{BJM}
{ \scshape Biringer, Ian; Johnson, Jesse; Minsky, Yair. }  Extending pseudo-Anosov maps into compassion bodies.  Preprint (2010).  arXiv:1011.0021v1

\bibitem{CB}
{ \scshape Casson, Andrew J.; Bleiler, Steven A. }  Automorphisms of Surfaces after Nielsen and Thurston, vol. 9 of \emph{London Mathematical Society Student Texts}.  \emph{Cambridge University Press, Cambridge}, 1988

\bibitem{CG}
{ \scshape Casson, A. J.; Gordon, C. McA. }  A loop theorem for duality spaces and fibered ribbon knots.
\emph{Inventiones. mathematicae}, {\bf 74 } (1983), no. 1, 119 - 137

\bibitem{CL}
{ \scshape Casson, A.J.; Long, D.D. } Algorithmic compression of surface automorphisms.  \emph{Inventiones mathematicae}, { \bf 81 } (1985) 295 - 303

\bibitem{L}
{ \scshape Long, D.D. } Bounding laminations. \emph{Duke Mathematical Journal}, { \bf 56 } (1) (1988)

\bibitem{L2}
{ \scshape Long, D.D. } Planar kernels in surface groups.  \emph{Oxford Quart. J. Math.}, { \bf 35 } (2) (1984) 305 - 310

\end{thebibliography}
\end{document}